\documentclass[11pt,reqno]{amsart}
\usepackage{amsthm,amsmath,amssymb}
\usepackage{graphicx}
\usepackage{float}
\usepackage[colorlinks=true,
linkcolor=blue,
urlcolor=blue,
citecolor=blue]{hyperref}
\usepackage{mathtools}
\usepackage{relsize}
\usepackage{ytableau}
\usepackage{tikz}
\usepackage{enumerate}
\usepackage[toc,page]{appendix}
\usepackage{lscape}
\usepackage{multirow}
\usepackage{adjustbox,expl3,etoolbox} 
\usepackage[margin = 3.1cm]{geometry}
\usepackage{diagbox}
\usepackage{xcolor}
\usepackage{comment}
\newtheorem{thm}{Theorem}[section]
\newtheorem{lem}[thm]{Lemma}

\usepackage{mathpple}
\usepackage{times}

\newtheorem{prop}[thm]{Proposition}
\newtheorem{conj}[thm]{Conjecture}
\newtheorem{cor}[thm]{Corollary}
\newtheorem{exa}[thm]{Example}
\newtheorem{rmk}[thm]{Remark}

\newtheorem{qst}[thm]{Question}
\newtheorem{claim}{Claim}

\newtheorem{ass}{Assumption}
\theoremstyle{definition}

\DeclareMathOperator\Aut{Aut}

\DeclareMathOperator\sym{Sym}
\DeclareMathOperator\alt{Alt}
\newcommand{\agl}[2]{\operatorname{AGL}_#1(#2)}

\newcommand{\psl}[2]{\operatorname{PSL}_#1(#2)}

\newcommand{\pgammal}[2]{\operatorname{P\Gamma L}_#1(#2)}

\newcommand{\mathieu}[1]{\operatorname{M}_{#1}}

\newcommand{\pg}[2]{\operatorname{PG}_{#1}(#2)}
\newcommand{\dih}[1]{\operatorname{D}_{#1}}

\newcommand{\itbf}[1]{{\bf{{\emph{{#1}}}}}}

\letcs\replicate{prg_replicate:nn}

\begin{document}	
	
	\title[]{The Erd\H{o}s-Ko-Rado Theorem for non-quasiprimitive groups of degree $3p$}

	\subjclass[2010]{Primary 05C35; Secondary 05C69, 20B05}
	
	\keywords{Derangement graphs, cocliques, projective special linear groups}
	
	\date{\today}
	
	\maketitle
	\begin{center}
		{\sc Roghayeh Maleki \quad\qquad Andriaherimanana Sarobidy\\ \hspace*{4.5cm} Razafimahatratra\textsuperscript{*}}
	\end{center}
	\vspace*{0.4cm}
	\begin{center}
		{\Small\it University of Primorska, UP FAMNIT, Glagolja\v{s}ka 8, 6000 Koper, Slovenia\\
		University of Primorska, UP IAM, Muzejski trg 2, 6000 Koper, Slovenia\\
		\textsuperscript{*} Corresponding author
		}
	\end{center}
	
	\begin{abstract}
		The \itbf{intersection density} of a finite transitive group $G\leq \sym(\Omega)$ is the rational number $\rho(G)$ given by the ratio between the maximum size of a subset of $G$ in which any two permutations agree on some elements of $\Omega$ and the order of a point stabilizer of $G$. In 2022, Meagher asked whether $\rho(G)\in \{1,\frac{3}{2},3\}$ for any transitive group $G$ of degree $3p$, where $p\geq 5$ is an odd prime. For the primitive case, it was proved in [\emph{J. Combin. Ser. A}, 194:105707, 2023] that the intersection density is $1$. 
		
		It is shown in this paper that the answer to this question is affirmative for non-quasiprimitive groups, unless possibly when $p = q+1$ is a Fermat prime and $\Omega$ admits a unique $G$-invariant partition $\mathcal{B}$ such that the induced action $\overline{G}_\mathcal{B}$ of $G$ on $\mathcal{B}$ is an almost simple group containing $\psl{2}{q}$. 
	\end{abstract}
	
	\section{Introduction}
	
	The Erd\H{o}s-Ko-Rado (EKR) Theorem \cite{erdos1961intersection} is a fundamental theorem in extremal set theory. For any two positive integers $n\geq k\geq 1$, we let $\binom{[n]}{k}$ be the collection of all $k$-subsets of $\{1,2,\ldots,n\}$. The EKR Theorem is stated as follows.
	\begin{thm}[Erd\H{o}s-Ko-Rado]
		Let $n\geq k \geq 1$ be two positive integers such that $n\geq 2k$. If $\mathcal{F} \subset \binom{[n]}{k}$ such that $A\cap B \neq \varnothing$ for all $A,B\in \mathcal{F}$, then
		\begin{align*}
			\left|\mathcal{F}\right|\leq \binom{n-1}{k-1}.
		\end{align*}
		Moreover, if $n\geq 2k+1$ then equality holds if and only if $\mathcal{F}$ is the collection of all $k$-subsets of $\{1,2,\ldots,n\}$ containing a prescribed element.
	\end{thm}
	
	The EKR Theorem has been extended to various combinatorial objects throughout the years. See the monograph \cite{godsil2016erdos} on EKR type results for more details. This paper is concerned with the extension of the EKR Theorem to transitive permutation groups of fixed degrees.
	
	All groups considered in this paper are finite and all graphs are simple and undirected. Given a transitive group $G\leq \sym(\Omega)$, we say that $\mathcal{F} \subset G$ is \itbf{intersecting} if for any $g,h\in \mathcal{F}$, there exists $\omega\in \Omega$ such that $\omega^g = \omega^h$. For any $\omega\in \Omega$, the stabilizer $G_\omega = \{ g\in G : \omega^g = \omega \}$ of $\omega$ and its cosets are obvious intersecting sets of $G$. The \itbf{intersection density} of the transitive group $G$ is 
	\begin{align*}
		\rho(G):= \frac{1}{|G_\omega|}\max\left\{ |\mathcal{F}|:\ \mathcal{F} \subset G \mbox{ is intersecting} \right\}.
	\end{align*}
	We note that $\rho(G)\geq 1$ since $G_\omega$ itself is intersecting. Transitive groups with intersection density equal to $1$ have been subject to a great deal of focus since the paper of Deza and Frankl \cite{Frankl1977maximum} in the late 70s. See for instance \cite{ahmadi2013erd,cameron2003intersecting,ernst_schmidt_2023,ellis2011intersecting,godsil2009new,larose2004stable,meagher2011erdHos,spiga2019erdHos}.
	
	Recently, works on transitive groups with intersection density larger than $1$ have appeared in the literature. In particular, the paper \cite{li2020erd} by Li, Song and Pantangi, and \cite{meagher180triangles} by Meagher, Spiga and the second author explored the theory of the transitive groups with intersection density larger than $1$. In \cite{meagher180triangles}, the following conjectures were posed.
	\begin{conj}
		Let $G\leq \sym(\Omega)$ be a transitive group.
		\begin{enumerate}
			\item If $|\Omega|$ is a prime power, then $\rho(G) = 1$.\label{first}
			\item If $|\Omega| = 2p$, where $p$ is an odd prime, then $1\leq \rho(G) \leq 2$.\label{second}
			\item If $|\Omega| = pq$, where $p$ and $q$ are two odd primes, then $\rho(G) = 1$.\label{third}
		\end{enumerate}\label{conj}
	\end{conj}
	
	Conjecture~\ref{conj}\eqref{first} was  proved independently in \cite{hujdurovic2021intersection} and \cite{li2020erd}, and Conjecture~\ref{conj}\eqref{second} was proved in \cite{AMC2554}. In \cite{hujdurovic2022intersection}, Conjecture~\ref{conj}\eqref{third} was disproved by constructing a family of transitive groups of degree $pq$, where $p = \frac{q^r-1}{q-1}$, whose intersection density is equal to $q$. Though the third conjecture was disproved, it is still of interest to know all the possible intersection densities of transitive groups of degree a product of two distinct odd primes. To this end, we recall the following set which was first defined in \cite{meagher180triangles}
	\begin{align*}
		\mathcal{I}_n := \left\{ \rho(G) : G \leq \sym(\Omega)\mbox{ is transitive with } |\Omega|=n \right\},
	\end{align*}
	for $n\geq 2$. For instance, one can verify using \verb*|Sagemath| \cite{sagemath} that $\mathcal{I}_{15} = \{1\}$, whereas $\mathcal{I}_{10} = \{1,2\}$.
	
	Let $G\leq \sym(\Omega)$ be a finite transitive group. A $G$\itbf{-invariant partition}, or \itbf{system of imprimitivity},  or a \itbf{complete block system} of $G$ is a partition that is preserved by the action of $G$. That is, given a partition $\mathcal{B}$ of $\Omega$, either $B^g = B$ or $B^g \cap B=\varnothing$, for any $B\in \mathcal{B}$ and $g\in G$. The elements of a $G$-invariant partition of $G$ are called \itbf{blocks}. The trivial partitions of $G$ are $\{\Omega\}$ and $\{ \{\omega\} : \omega \in \Omega \}$. If the only $G$-invariant partitions of $G$ are the trivial ones, then we say that $G$ is \itbf{primitive}; otherwise, it is \itbf{imprimitive}. The group $G\leq \sym(\Omega)$ is \itbf{quasiprimitive} if any non-trivial subgroup of $G$ is transitive. The transitive groups on $\Omega$ can be subdivided into three categories: the primitive groups, the quasiprimitive imprimitive groups, and the non-quasiprimitive imprimitive groups. Let us define the following subsets of $\mathcal{I}_n$, for $n\geq 1$, with respect to these categories
	\begin{align*}
		\mathcal{P}_n &:= \left\{ \rho(G):\ G \mbox{ is primitive of degree $n$} \right\},\\
		\mathcal{Q}_{n} &:= \left\{ \rho(G) :\ G \mbox{ is quasiprimitive and imprimitive of degree $n$} \right\} ,\\
		\mathcal{NQ}_n&:= \left\{ \rho(G) :\ G\mbox{ is non-quasiprimitive and imprimitive of degree $n$} \right\}.
	\end{align*}

	In \cite{razafimahatratra2021intersection}, most of the primitive groups of degree a product of two distinct odd primes were shown to have intersection density equal to $1$. In \cite{behajaina2022intersection}, some results on the intersection density of imprimitive groups of degree a product of two odd primes were also proved. This paper is part of a larger project whose aim is to determine the exact set $\mathcal{I}_{pq}$ for any two distinct of primes $p$ and $q$. We restrict ourselves to the study of the set $\mathcal{I}_{3p}$, where $p\geq 5$ is an odd prime. The motivation for this work is the following question due to Meagher.
	\begin{qst}
		Let $p\geq 5$ be an odd prime. Is it true that $\mathcal{I}_{3p}\subseteq \{1,\frac{3}{2},3\}$?\label{qst:karen}
	\end{qst}
	
	In \cite{razafimahatratra2021intersection}, it was shown that $\mathcal{P}_{3p}=\{ 1 \}$, for any prime $p\geq 5$. Therefore, we only need to consider the imprimitive groups of degree $3p$. If $G\leq \sym(\Omega)$ is imprimitive and non-quasiprimitive, then there exists a non-trivial subgroup $N\trianglelefteq G$ which is intransitive. The set $\mathcal{B}$ of orbits of $N$ forms a $G$-invariant partition of $\Omega$ and the induced action, denoted by $\overline{G}_{\mathcal{B}}$, of $G$ on these orbits is transitive of degree $p$, so it is either solvable or $2$-transitive. By \cite{hujdurovic2022intersection} and \cite{razafimahatratra2021intersection}, if $\Omega$ has more than one non-trivial $G$-invariant partitions or it admits a $G$-invariant partitions whose blocks are of size $p$, then $\rho(G) = 1$. Hence, we can always assume that $\mathcal{B}$ is the unique $G$-invariant partition of $G$.
	
	 In this paper, we prove the following theorem.
	\begin{thm}
		Let $G\leq \sym(\Omega)$ be transitive of degree $3p$, $1\neq N\trianglelefteq G$ be intransitive, and $\overline{G}$ be the induced action of $G$ on the $G$-invariant partitions of $\Omega$ consisting of orbits of $N$. One of the following occurs.
		\begin{enumerate}
			\item If $\overline{G}$ is solvable, then $\rho(G) \in \{1,\frac{3}{2},3\}$.
			\item If $\overline{G}$ is $2$-transitive, then $\rho(G)\in \{1,\frac{3}{2},3\}$ or $\overline{G}$ is an almost simple group containing $\psl{2}{q}$ with $p=q+1$ a Fermat prime. 
		\end{enumerate}
		In particular, if $p\geq 5$ is not a Fermat prime, then $ \mathcal{NQ}_{3p} \subset \{ 1,\frac{3}{2},3 \}$.\label{thm:main}
	\end{thm}
	
	\subsection*{Strategy of the proof}
	Our proof relies heavily on a characterization of groups with intersection density larger than $1$ in \cite{behajaina2022intersection}. Suppose that $\mathcal{B} = \{B_1,B_2,\ldots,B_p\}$ is the unique $G$-invariant partition of $\Omega$. We will see that in our case an imprimitive group is non-quasiprimitive if and only if the kernel $\ker(G\to \overline{G}) $ of the canonical homomorphism $G\to \overline{G}$ is non-trivial. It was shown in \cite{behajaina2022intersection} that the only possible groups giving intersection density larger than $1$ in the class of non-quasiprimitive imprimitive groups are those with the property that $\ker(G\to\overline{G})$ is derangement-free. 
	
	The analysis can then be divided into two cases, depending on whether $\overline{G}$ is solvable, that is, $\overline{G} = \langle \alpha\rangle \rtimes \langle\beta\rangle \leq \agl{1}{p}$ or $2$-transitive. We will show that if $\overline{G}$ is solvable, then the analysis can be further subdivided into two subcases depending on whether $\ker(G\to\overline{G})$ admits an involution or not. If $\ker(G\to\overline{G})$ has no involutions, then the derangement graph of $G$ (see Section~\ref{sect:der-graph}) is a lexicographic product of a graph from a family of graphs defined in Section~\ref{sect:graph-theory} and an empty graph. From this we can show that the intersection density is in $\{1,\frac{3}{2},3\}$. Using the graph in Section~\ref{sect:graph-theory} along with the No-Homomorphism Lemma, we obtain that the intersection density also belongs to $\{1,\frac{3}{2},3\}$ for the case where $\ker(G\to \overline{G})$ has an involution. If $\overline{G}$ is $2$-transitive, then we show that $G$ must contain a transitive subgroup $H$ with similar properties to $G$ except that $\overline{H}$ is solvable.  We obtain an upper bound equal to $1$, except when $p$ is a Fermat prime and $\overline{G}$ is an almost simple group containing a projective special linear group of degree $2$, in which case we only get an upper bound equal to $\frac{3}{2}$.
	
	The proof of Theorem~\ref{thm:main} follows from Theorem~\ref{thm:special-case}, Theorem~\ref{thm:special-case-2}, Theorem~\ref{thm:special-case-3}, and Theorem~\ref{thm:special-case-4}.
	\subsection*{Organization of the paper}
	In Section~\ref{sect:permutation-groups}, we give some necessary background results from permutation group theory. In Section~\ref{sect:graph-theory}, we define a family of graphs that are crucial to the proof of the main result.  In Section~\ref{sect:analysis}, we give an analysis of the cases to consider. Then, the solvable case is proved in Section~\ref{sect:solvable}, Section~\ref{sect:no-involutions}, and Section~\ref{sect:involutions}. In Section~\ref{sect:2-transitive}, we prove the $2$-transitive case. We give some open problems regarding the remaining cases in Section~\ref{sect:conclusion}.

	\section{Background results on permutation group theory}\label{sect:permutation-groups}
	
	\subsection{Basic notions}
	 Henceforth, we let $G\leq \sym(\Omega)$ be a transitive group.
	An $(m,n)$\itbf{-semi-regular} element of $G$ is a permutation which is a product of $n$ cycles of length $m$. If $m$ and $n$ are clear from the context, then we just use the term \itbf{semi-regular} element. A \itbf{semi-regular subgroup} of $G$ is a subgroup $H$ with the property that for any two elements $\omega,\omega^\prime \in \Omega$, there exists at most one element $h\in H$ such that $\omega^\prime = \omega^h$. A subgroup generated by a semi-regular element is clearly a semi-regular subgroup. 
	
	We say that $G\leq \sym(\Omega)$ is \itbf{$2$-transitive} or \itbf{doubly transitive} if for any pairs of elements $(\omega_1,\omega_2),(\omega_3,\omega_4) \in \Omega\times \Omega$ such that $\omega_1\neq \omega_2$ and $\omega_3\neq \omega_4$, there exists $g\in G$ such that $(\omega_1,\omega_2)^g = (\omega_3,\omega_4)$. In other words, $G$ is transitive on $\Omega^{(2)} =\{ (\omega,\omega^\prime)\in \Omega \times \Omega: \omega\neq \omega^\prime \}$.

	For any $G$-invariant partition $\mathcal{B}$ of a transitive group $G\leq \sym(\Omega)$, we may define $\overline{G}_{\mathcal{B}} = \{ \overline{g} : g\in G \}$, where $\overline{g}$ is the permutation of $\mathcal{B}$ induced by $g$. As the groups that we study in this paper admit a unique non-trivial $G$-invariant partition, we will use the notation $\overline{G}$ instead of $\overline{G}_{\mathcal{B}}$. Clearly, $\overline{G}\leq \sym(\mathcal{B})$ is transitive. Consequently, $G$ acts on $\mathcal{B}$ via the natural group homomorphisms $G\to \overline{G} \to \sym(\mathcal{B})$. We define $\ker(G \to \overline{G})$ to be the kernel of the induced action of $G$ on $\mathcal{B}$. By the First Isomorphism Theorem, we have $G/\ker(G\to\overline{G}) \cong \overline{G}$, and so we note that the action of $G$ on $\mathcal{B}$ corresponds to a permutation group of $\sym(\mathcal{B})$ if and only if $\ker(G\to \overline{G})$ is trivial. 
	
	Given $g\in G$ and $B \subset \Omega$ such that $B^g = B$, we let $g_{|B}$ be the restriction of the permutation $g \in \sym(\Omega)$ onto $B$. We will denote the order of $g\in G$ by $o(g)$.
	
	\subsection{Transitive groups of prime order}	
	Throughout this subsection, we assume that $G\leq \sym(\Omega)$ is transitive of prime degree $p$. Recall that the socle of a group is the subgroup generated by its minimal normal subgroups. Let $\operatorname{Soc}(G)$ denote the socle of $G$.
	
	It was already known in the late 1800s due to Burnside that a transitive group of prime degree had to be solvable or $2$-transitive. The Classification of Finite Simple Groups (CFSG) made it possible to obtain important classification results in permutation group theory. One of such classifications is that of the $2$-transitive groups. A transitive group $G$ of degree $p$ that is $2$-transitive has the property $\operatorname{Soc}(G)\leq G\leq \Aut(\operatorname{Soc}(G))$. The possibilities for the socles of $G$ are
	\begin{enumerate}[(a)]
		\item $\operatorname{Soc}(G) = C_p$, in which case $G = \agl{1}{p}$,\label{trans1}
		\item $\operatorname{Soc}(G) = \alt(p)$,\label{trans2}
		\item $p = 11$ and $G = \operatorname{Soc}(G) = \mathieu{11}$ or $G = \operatorname{Soc}(G) = \psl{2}{11},$,\label{trans3}
		\item $p = 23$ and $G = \operatorname{Soc}(G) = \psl{2}{23}$ or $G = \operatorname{Soc}(G) = \mathieu{23}$,\label{trans4}
		\item $p = \frac{q^n-1}{q-1}$ and $\operatorname{Soc}(G)= \psl{n}{q}$, for some prime power $q$.\label{trans5}
	\end{enumerate}
	For \eqref{trans5}, since $p = \frac{q^n-1}{q-1}$ is a prime, it is not hard to show that in fact $n$ itself is a prime. If $n = 2$, then since $p = q+1$ is an odd prime, we must have that $q$ is an even power of $2$ and $p$ is therefore a Fermat prime.
	
	\section{Graph Theory}\label{sect:graph-theory}
	
	Given a graph $X=(V,E)$, we use the notation $x\sim_X y$ to represent the fact that $\{x,y\} \in E$, or equivalently, $x$ and $y$ are adjacent. A \itbf{clique} in $X$ is a subset of vertices in which any two are adjacent. A \itbf{coclique} in $X$ is a subset of vertices in which no two are adjacent. The maximum size of a clique and a coclique in a graph $X$ are denoted respectively by $\omega(X)$ and $\alpha(X)$.

	Let $X$ and $Y$ be two graphs. The \itbf{lexicographic product} $X[Y]$ of $X$ and $Y$ (in this order) is the graph with vertex set $V(X)\times V(Y)$ such that two vertices $(x,y)$ and $(x^\prime,y^\prime)$ in $V(X)\times V(Y)$ are adjacent if and only if 
	\begin{align*}
		\begin{cases}
			x\sim_X x^\prime, \mbox{ or }\\
			x = x^\prime \mbox{ and } y\sim_Y y^\prime.
		\end{cases}
	\end{align*}
	
	Let $m,n\geq 2$, $k$ and $r$ be positive integers such that $k\mid r$. For any positive integer $t$, we let $[t]:= \left\{ 1,2,\ldots,t \right\}$. Define the graph $\Gamma_{m,n}^k(r)$ whose vertex set is 
	\begin{align*}
		V = \left\{ (a,b,c): a\in [r], b\in [m],c\in [n]\right\}.
	\end{align*}
	Let $\pi = \{ P_1,P_2,\ldots,P_k \}$ be a uniform partition of $[r]$, i.e., a partition whose parts are of equal size. The edge set of $\Gamma_{m,n}^k(r)$ is defined in a way that
	\begin{align*}
		(a,b,c)\sim (a^\prime,b^\prime,c^\prime) \Leftrightarrow
		\begin{cases}
			c = c^\prime \mbox{ and } b\neq b^\prime, \mbox{ or }\\
			c\neq c^\prime \mbox{ and $a$ and $a^\prime$ are in different parts of $\pi$.} 
		\end{cases}
	\end{align*}
	We note that $\Gamma^k_{m,n}(r)$ is independent of the uniform partition $\pi$ since using a different uniform partition with $k$ blocks yields an isomorphic graph. Hence, we fix a uniform partition $\pi =\{P_1,\ldots,P_k\}$.
	
	\begin{exa}
		The graph in Figure~\ref{fig2} is $\Gamma_{3,2}^3(6)$. If the edge between two blobs is black, then the corresponding induced subgraph is $K_{6,6}$, and if it is red, then the induced subgraph is $X[\overline{K}_6]$, where $X$ is the graph define in Figure~\ref{fig1}.
		\begin{figure}[t]
			\centering
			\includegraphics[width=7cm]{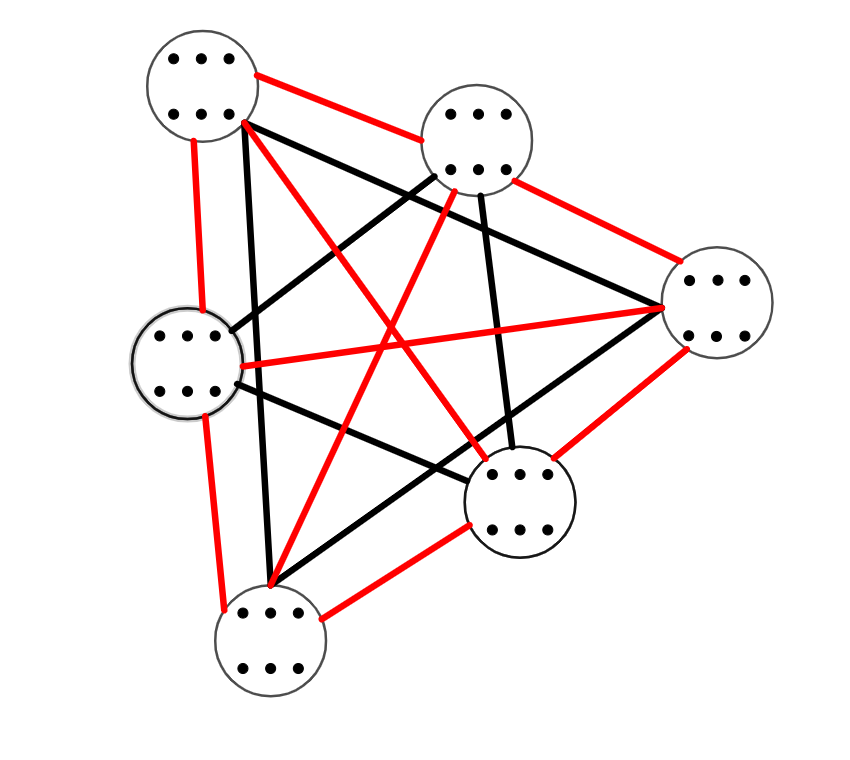}
			\caption{The graph $\Gamma_{3,2}^3(6)$}\label{fig2}
		\end{figure}
	\end{exa}
	
	\begin{figure}[b]
		\includegraphics[width=5cm]{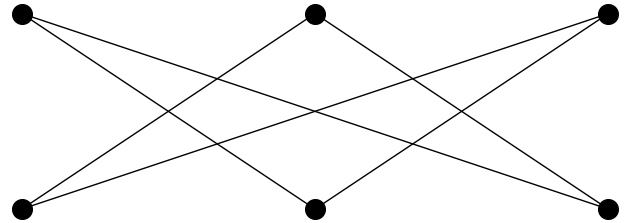}
		\caption{The graph $K_{3,3}$ minus a perfect matching.}\label{fig1}
	\end{figure}
	For any fixed $b\in [m]$ and $c\in [n]$, define $U_{b,c}(r):=\{ (a,b,c): a\in [r] \}$. We note that the set $U_{b,c}(r)$ is a coclique of $\Gamma_{m,n}^k(r)$. Let us now analyze the edges between these cocliques in the graph $\Gamma_{m,n}^k(r)$. We omit the proof of the next proposition since it follows directly from the definition of the edge set.
	\begin{prop}
		Let $b,b^\prime \in [m]  $, and $c,c^\prime \in [n]$ such that $(b,c) \neq (b^\prime,c^\prime)$. Then, one of the following holds.
		\begin{enumerate}
			\item If $c=c^\prime$, then the subgraph of $\Gamma_{m,n}^k(r)$ induced by $U_{b,c}(r) \cup U_{b^\prime,c^\prime}(r)$ is a complete bipartite graph $K_{r,r}$.\label{prop:complete-mult}
			\item If $c\neq c^\prime$, then the subgraph of $\Gamma_{m,n}^k(r)$ induced by $U_{b,c}(r)\cup U_{b^\prime,c^\prime}(r)$ is the lexicographic product $\widetilde{K}_{k,k}[\overline{K_{\frac{r}{k}}}]$, where the graph $\widetilde{K}_{k,k}$ is the graph obtained by removing a perfect matching from the complete bipartite graph $K_{k,k}$.
		\end{enumerate}\label{prop:set}
	\end{prop}
	
	The following is an immediate corollary of the above proposition. 
	\begin{cor}
		We have the graph isomorphism $\Gamma_{m,n}^k(r) \cong \Gamma_{m,n}^k(k)\left[\overline{K_{\frac{r}{k}}}\right]$, where $\overline{K_{\frac{r}{k}}}$ is the complement of the complete graph $K_{\frac{r}{k}}$.
	\end{cor}
	
	\begin{lem}
		The independence number of $\Gamma_{m,n}^k(k)$ is equal to $\max\{ k,n \}$.\label{lem:max-coclique}
	\end{lem}
	\begin{proof}
		Let $\mathcal{F}$ be a coclique of $\Gamma_{m,n}^k(k)$. If $\mathcal{F} \subset U_{1,1}(k)$, then $|\mathcal{F}|\leq k$. Thus, we may assume without loss of generality that $\mathcal{F} \cap U_{1,1}(k)\neq \varnothing$ and that $\mathcal{F} \setminus U_{1,1}(k) \neq \varnothing.$ Let $z\in \mathcal{F} \cap U_{1,1}(k) =  \left\{(a,1,1): a = j\right\}$, for some $j \in [k]$. Let us decompose $\mathcal{F}$ into $$\mathcal{F} = \mathcal{F}_1\cup \mathcal{F}_2\cup \ldots \cup \mathcal{F}_{n},$$ where $\mathcal{F}_i = \mathcal{F} \cap \left( U_{1,i}(k)\cup U_{2,i}(k)\cup \ldots\cup U_{m,i}(k) \right)$, for any $i\in \{1,2,\ldots,n\}$. Then, for any $i \in \{2,3,\ldots,n\}$,  the vertex $z$ is non-adjacent to the vertex $ (j,b,i)$, for any $b\in [m]$. Moreover, $z$ is adjacent to all vertices in $\{(i^\prime,b,i) :i^\prime \neq j \}$, for all $b\in [m]$. By Proposition~\ref{prop:set}~\eqref{prop:complete-mult}, if $\mathcal{F} \cap U_{s,i}(k) \neq \varnothing$, then $\mathcal{F} \cap \left( U_{1,i} \cup \ldots \cup U_{s-1,i} \cup U_{s+1,i} \cup \ldots \cup U_{m,i} \right) = \varnothing$. Hence, we conclude that $|\mathcal{F}_{i}|\leq 1$. By a similar, argument, we also show that $|\mathcal{F}_i|\leq 1$, for any $i\in \{ 1,2,\ldots,n \}$. Consequently, we have that
		\begin{align*}
			\alpha (\Gamma_{m,n}^k(k))\leq \max\{ k,n \}.
		\end{align*}
		It is obvious that the two values in the upper bound is attained by $U_{1,1}(r)$ or by $\{ (1,1,c):  c\in [n] \}$.
	\end{proof}
	\begin{cor}
		We have $\alpha\left( \Gamma_{m,n}^k(k) \right) = \max \left\{ r,\frac{nr}{k} \right\}$.
	\end{cor}
	\begin{proof}
		The result follows from the fact that $\alpha\left( \Gamma_{m,n}^k(k) \left[\overline{K}_{\frac{r}{k}}\right] \right) = \alpha(\Gamma_{m,n}^k(k))\alpha \left(\overline{K_{\frac{r}{k}}}\right)$.
	\end{proof}
	
	Now, we define another graph which is similar to $\Gamma_{m,n}^k(r)$ by introducing a set of permutations $\Sigma$. Let $\pi = \{ P_1,P_2,\ldots,P_k \}$ be any uniform partition of $[r]$ into $k$ parts. For any $b,b^\prime \in [m]$ and $c,c^\prime \in [n]$, let  $\sigma_{b,b^\prime}^{c,c^\prime} \in \sym(k)$ be a permutation depending only on $b,b^\prime,c,c^\prime$, and define the multiset of permutations 
	\begin{align*}
		\Sigma = \left\{ \sigma_{b,b^\prime}^{c,c^\prime}: b,b^\prime \in [m] ,\ c,c^\prime \in [n] \right\}.
	\end{align*}
	 Now, we define the graph $\Gamma_{m,n}^{k,\Sigma}(r)$ to be the graph whose vertex set is 
	 \begin{align*}
	 	V = \left\{ (a,b,c): a\in [r], b\in [m],c\in [n]\right\}.
	 \end{align*}
	 Two elements $(a,b,c),(a^\prime,b^\prime,c^\prime) \in V$ are adjacent if and only if
	 \begin{align*}
	 	\begin{cases}
	 		c = c^\prime \mbox{ and } b\neq b^\prime, \mbox{ or }\\
	 		c\neq c^\prime \mbox{ and $a \in P_i$ and $a^\prime \in P_j$ such that $\sigma_{b,b^\prime}^{c,c^\prime}(i) \neq j$.} 
	 	\end{cases}
	 \end{align*}
	 If the $\sigma_{b,b^\prime}^{c,c^\prime}$ is the identity map for all $b,b^\prime \in [m]$ and $c,c^\prime \in [n]$, then $\Gamma_{m,n}^{k,\Sigma}(r) = \Gamma_{m,n}^k(r)$.
	 
	The graphs $\Gamma_{m,n}^k(r)$ and $\Gamma_{m,n}^{k,\Sigma}(r)$ are not necessarily isomorphic, but they are locally isomorphic in the sense that for all $b,b^\prime \in [m]$ and $c,c^\prime \in [n]$, the subgraph induced by $\{ (a,b,c): a\in [r] \} \cup \{ (a,b^\prime,c^\prime): a\in [r] \}$ are isomorphic. A direct consequence of this is that the independence number of $\Gamma_{m,n}^k(r)$ and $\Gamma_{m,n}^{k,\Sigma}(r)$ are equal. In fact, we have
	\begin{align*}
		\Gamma_{m,n}^{k,\Sigma}(r) \cong \Gamma_{m,n}^k(k) \left[ \overline{K_{\frac{r}{k}}} \right].
	\end{align*}

	 We omit the proof of the next lemma since it is similar to that of Lemma~\ref{lem:max-coclique} and its corollary. 
	
		\begin{lem}
		The independence number of $\Gamma_{m,n}^{k,\Sigma}(r)$ is equal to $\max\{ r,\frac{rn}{k} \}$.\label{lem:max-coclique-gen}
	\end{lem}
	
	We end this section by stating the No-Homomorphism Lemma. We recall that a homomorphism between a graph $X$ and a graph $Y$ is a map from the vertex set of $X$ to the vertex set of $Y$ which maps an edge to an edge.
	\begin{lem}[No-Homomorphism Lemma \cite{albertson1985homomorphisms}]
		Let $X$ be a graph and $Y$ be a vertex-transitive graph. If there is a graph homomorphism from $X$ to $Y$, then 
		\begin{align*}
			\frac{\alpha(Y)}{|V(Y)|} \leq \frac{\alpha(X)}{|V(X)|}.
		\end{align*} 
	\end{lem}
	
		\section{Derangement graphs}\label{sect:der-graph}
	Let $H$ be a group and $C\subset H\setminus \{1\}$. Recall that the \itbf{Cayley digraph} $\operatorname{Cay}(H,C)$ is the digraph with vertex set equal to $H$, and for any $x,y\in H$, $(x,y)$ is an arc if and only if $yx^{-1}\in C$. If $x^{-1} \in C$ whenever $x\in C$, then $\operatorname{Cay}(H,C)$ is a simple undirected graph. It is not hard to see that $\operatorname{Cay}(H,C)$ is vertex-transitive since its automorphism group contains a regular subgroup given by the right-regular-representation of $H$. 
	
	Now, let $G\leq \sym(\Omega)$ be a transitive group. Recall that a \itbf{derangement} of $G$ is a fixed-point-free permutation. A famous result of Jordan asserts that a finite transitive group always admits a derangement \cite{jordan1872recherches}. Let $\operatorname{Der}(G)$ be the set of all derangements of $G$. The \itbf{derangement graph} of $G$ is the Cayley graph $\Gamma_{G}:=\operatorname{Cay}(G,\operatorname{Der}(G))$.
	
	The derangement graph is important in the analysis of the intersection density of the transitive group $G\leq \sym(\Omega)$. Indeed, if $\mathcal{F}\subset G$ is intersecting, then for any $g,h\in \mathcal{F}$, $hg^{-1}$ fixes an element of $\Omega$. Hence, $hg^{-1}$ is not in $\operatorname{Der}(G)$, and thus $g$ and $h$ are not adjacent in $\Gamma_{G}$. In other words, $\mathcal{F}$ is a coclique of $\Gamma_{G}$. Conversely, if $\mathcal{F}$ is a coclique of $\Gamma_{G}$, then any two permutations $g,h\in \mathcal{F}$ are such that $hg^{-1}\not\in \operatorname{Der}(G)$, i.e., $g$ and $h$ agree on some element of $\Omega$. Therefore, we conclude that $\mathcal{F}\subset G$ is intersecting if and only if $\mathcal{F}$ is a coclique in $\Gamma_{G}$. From this correspondence, we derive that
	\begin{align*}
		\rho(G) = \frac{\alpha(\Gamma_{G})}{|G_\omega|}.
	\end{align*}
	
	Since the derangement graph is regular and vertex transitive, we can use various techniques to get an upper bound on the independence number of $\Gamma_{G}$. The next result will be useful to the results in this paper.
	\begin{lem}[Clique-coclique bound \cite{godsil2016erdos}]
		Let $X = (V,E)$ be a vertex-transitive graph. Then, 
		\begin{align*}
			\alpha(X)\omega(X)\leq |V|.
		\end{align*}
		Moreover, if equality holds, then for any coclique of maximum size $S$ and for any clique of maximum size $T$, we have $|S\cap T| = 1$.%
	\end{lem}
	We derive the following corollary from this lemma.
	\begin{cor}
		If $G\leq \sym(\Omega)$ is transitive, then $\rho(G)\leq \frac{|\Omega|}{\omega(\Gamma_{G})}$. In particular, if $\Gamma_{G}$ has a clique of size $|\Omega|$, then $\rho(G) = 1$.\label{cor:clique-coclique}
	\end{cor}
	\begin{cor}
		If $G\leq \sym(\Omega)$ is transitive of degree $3p$, where $p\geq 5$ is an odd prime, then $\rho(G)\leq 3$.\label{cor:upper-bound-density}
	\end{cor}
	\begin{proof}
		In \cite{maruvsivc1981vertex} Maru\v{s}i\v{c} showed that any transitive group of degree $mq$, where $q$ is a prime and $m\leq q$,  admits a semi-regular element of order $q$. Let $g\in G$ be a semi-regular element of order $p$. The subgroup $\langle g\rangle$ is a clique of size $p$ in the derangement graph $\Gamma_{G}$. By Corollary~\ref{cor:clique-coclique}, we have $\rho(G)\leq \frac{3p}{p} = 3.$
	\end{proof}
	
	\section{Analysis of the intersection density of groups of degree $3p$}\label{sect:analysis}

	Let $G\leq \sym(\Omega)$ be transitive and $|\Omega| = 3p$, where $p\geq 5$ is an odd prime. We recall the following lemma.
	\begin{lem}[\cite{hujdurovic2022intersection}]
		Let $G\leq \sym(\Omega)$ be a transitive group. If $\mathcal{B}$ is a $G$-invariant partition of $G$ and $H\leq G$ is a semi-regular subgroup whose orbits-partition is equal to $\mathcal{B}$, then $\rho(G) \leq \rho(\overline{G})$.\label{lem:quotient}
	\end{lem}

	If $G$ is primitive, then $\rho(G) = 1$ by \cite{razafimahatratra2021intersection}, so we may assume that $G$ is imprimitive. As $G$ is imprimitive of degree $3p$, it admits only blocks of size $3$ or $p$. If $G$ admits a block of size $p$ from a $G$-invariant partition $\mathcal{B}$ of $\Omega$, then we can also show that $\rho(G) = 1$. Indeed, Maru\v{s}i\v{c} showed in \cite{maruvsivc1981vertex} that a transitive group of degree $mp$, for $m\leq p$, admits a semi-regular element of order $p$. If $H$ is the semi-regular subgroup obtained from such a semi-regular element, it is straightforward to see that the set of orbits of $H$ must be equal to $\mathcal{B}$. By Lemma~\ref{lem:quotient}, we conclude that $\rho(G) \leq \rho(\overline{G}) = 1$, since $\overline{G}$ is transitive of prime degree.
	
	If $G$ admits at least two non-trivial $G$-invariant partitions, then $\rho(G) = 1$ by a result in \cite[Section~3]{razafimahatratra2021intersection}. Therefore, we may assume that $G$ admits a unique non-trivial $G$-invariant partition, 
	\begin{align}
		\mathcal{B} = \{ B_1,B_2,\ldots,B_p \}\label{eq:blocks}
	\end{align}
	where the blocks are $B_i = \{x_i,y_i,z_i\}$, for any $1\leq i\leq p$. We may distinguish two cases from hereon in our analysis.
	
	\subsection{The Quasi-primitive case}
	If $G\leq \sym(\Omega)$ is quasiprimitive (i.e., all of its non-trivial normal subgroups are transitive), then $\ker(G \to \overline{G})$ is trivial, since it cannot be transitive. Conversely, assume that $\ker(G\to \overline{G})$ is trivial. For any non-trivial normal subgroup $N \trianglelefteq G$, the orbits of $N$ form a $G$-invariant partition of $\Omega$ \cite[Theorem 1.6A]{dixon1996permutation} and so the orbits-partition of $N$ is either trivial or equal to $\mathcal{B}$. Since $N \neq 1$, the orbits-partition of $N$ cannot be $\{\{\omega\} : \omega \in \Omega\}$. Similarly, the orbits-partition of $N$ cannot be equal to  $\mathcal{B}$ since $N\neq 1$ and $N\leq \ker(G\to \overline{G})$ is trivial. Therefore, the orbits-partition of $N$ is equal to $\{\Omega\}$, and so $N$ is transitive.
	
	Thus, we have proved that $\ker(G\to \overline{G})$ is trivial if and only if $G$ is quasiprimitive. In this case, $G \cong \overline{G}$ and $G$ admits two faithful actions of degree $3p$ and $p$. Let $\omega \in \Omega$ and assume that $\omega \in B$, for some $B\in \mathcal{B}$. Let $G_{\{B\}}$ be the setwise stabilizer of the set $B$ in $G$. Note that $G_{\{B\}} = \overline{G}_{B}$ and if $g\in G_\omega$, then $g \in G_{\{B\}}$. Therefore,  we conclude that $G_\omega \leq G_{\{B\}}$. Hence, $G$ admits two actions, which are permutation equivalent to the actions of $G$ on cosets of $G_{\{B\}}$ (primitive) and on cosets of $G_\omega$ (imprimitive and quasiprimitive).
	
	Using the classification of transitive groups of prime degree, it is not hard to show (see \cite{behajaina2022intersection}) that $G$ is almost simple, and
	\begin{align*}
		\psl{n}{q} \leq G \leq \pgammal{n}{q}
	\end{align*}
	where $n\geq 1$ is an integer, $p = \frac{q^n-1}{q-1}$, and the point-stabilizer of $\psl{n}{q}$ in its action on the projective space $\pg{n-1}{q}$ admits a subgroup of index $3$.
	
	\subsection{The Genuinely imprimitive case}
	Assume that $N\trianglelefteq G$ is non-trivial and intransitive. As $N\trianglelefteq G$, its orbits-partition is a $G$-invariant partition of $G$, and since it is non-trivial and non-transitive, its orbits-partition is equal to $\mathcal{B}$. Consequently, $N\leq \ker(G\to \overline{G})$.
	
	If $N\trianglelefteq G$ is a minimal-normal subgroup of $G$, then either $N$ is transitive or it is intransitive and therefore contained in $\ker(G\to \overline{G})$. As $G$ is not quasiprimitive, note that there is always a minimal normal subgroup of $G$ contained in $\ker(G\to \overline{G})$. 
	
	Let $N\trianglelefteq G$ be a minimal normal subgroup of $G$ such that $N\leq \ker(G\to \overline{G})$. By \cite[Theorem~4.3A]{dixon1996permutation}, $N = T_1\times T_2\times \ldots \times T_k$, where $k\geq 1$ is a positive integer, $(T_i)_{i=1,\ldots,m}$ are simple normal subgroups of $N$ and  are conjugate in $G$. We may distinguish the cases where one of the factors (and therefore all, by conjugation) is abelian or not.
	\begin{itemize}
		\item If $T_1$ is non-abelian (and therefore all other $T_i$, for $1\leq i\leq k$), then it was proved in \cite{behajaina2022intersection} that $\rho(G) = 1$.
		\item If $T_1$ is abelian, then $N$ is an elementary abelian $3$-group \cite{behajaina2022intersection}. It was also proved in \cite{behajaina2022intersection} that if $\ker(G\to \overline{G})$ contains a derangement, then $\rho(G) = 1$. Consequently, the only transitive groups $G\leq \sym(\Omega)$ of interests in this case are those with the property that $\ker(G\to \overline{G})$ is derangement-free.
	\end{itemize}
	
	Consequently, we make the following assumption in the remainder of the paper.
	\begin{ass}
		Let $G\leq \sym(\Omega)$ be a transitive group of degree $3p$ admitting the $G$-invariant partition defined in \eqref{eq:blocks} as its only $G$-invariant partition. Assume that $\ker(G\to \overline{G})\neq 1$ is derangement-free and that any minimal normal subgroup of $G$ contained in $\ker(G \to \overline{G})$ is an elementary abelian $3$-group.\label{assumption}
	\end{ass}
	
	In the next sections, we will analyze the possible cases for the intersection density under Assumption~\ref{assumption}.
	
	\section{the solvable case}\label{sect:solvable}
	Let $G\leq \sym(\Omega)$ be a transitive group satisfying Assumption~\ref{assumption}. As $\overline{G}\leq \sym(\mathcal{B})$ is transitive of degree $p$, it is either $2$-transitive or solvable, and thus a proper subgroup of $\agl{1}{p}$. We assume further that $\overline{G}$ is solvable.
	
	Under these assumptions, we have that $\overline{G} = \langle \alpha\rangle \rtimes \langle \beta \rangle$, where $o(\alpha) = p$ and $o(\beta) = d \mid (p-1)$. Moreover, $\beta$ fixes $B_1$ and acts as a product of $\frac{p-1}{d}$ many $d$-cycles on $\mathcal{B} \setminus \{B_1\}$. In other words, $\overline{G}$ is a Frobenius group.
	
	Let $a\in G$ such that $\overline{a} = \alpha$. A result of Maru\v{s}i\v{c} in \cite{maruvsivc1981vertex} shows that the group $G$ of degree $3p$ always has a semi-regular element of order $p$. That is, we may assume that $a$ is a product of $3$ cycles of length $p$. 
	In other words, $o(a) = o(\alpha) = p$. Therefore, $$\left\langle\ker(G\to \overline{G}),a\right\rangle = \ker(G\to \overline{G})\langle a\rangle = \ker(G\to \overline{G}) \rtimes \langle a \rangle$$
	is a transitive group where $\ker(G\to\overline{G})$ is intersecting, since it is derangement-free \cite{meagher180triangles}. As the intersection density of a transitive subgroup of $G$ is at most $3$ by Corollary~\ref{cor:upper-bound-density}, it is not hard to see that the intersecting density of $\ker(G\to \overline{G}) \rtimes \langle \alpha\rangle$ is exactly equal to $3$. Indeed, $\ker(G\to \overline{G})$ is an intersecting set of size $\frac{|\ker(G\to \overline{G}) \rtimes \langle \alpha\rangle|}{p}$.

	Next, let $b\in G$ such that $\overline{b} = \beta$. Since $\overline{b} = \beta$, then $o(\beta)\mid o(b)$, so there exists a positive integer $r\geq 1$ such that $o(b)= dr$. Also, since $\overline{G} = \langle \alpha\rangle \rtimes \langle \beta\rangle$, there exists $t\in \{0,1,\ldots,p-1\}$ such that $\gcd(t,p) = 1$ and $\beta \alpha\beta^{-1} = \alpha^{t}$. Consequently, we have $\overline{bab^{-1}} = \overline{a^t}$, so there exists $h\in \ker(G\to \overline{G})$ such that
	\begin{align}
		bab^{-1} = ha^t.\label{eq:conj}
	\end{align}
	
	Define the group
	\begin{align}
		G(a,b) := \langle \ker(G\to \overline{G}),a,b \rangle.\label{second-group}
	\end{align}
	Since $\overline{G} = \langle \alpha \rangle \rtimes \langle \beta \rangle$, for any $g\in G$, there exists $i\in \{0,1,\ldots,p-1\}$ and $j\in \{0,1,\ldots,d-1\}$ such that $\overline{g} = \alpha^i\beta^j = \overline{a^ib^j}.$ Therefore, there exists $k\in \ker(G\to \overline{G})$ such that $g = ka^ib^j\in  \langle\ker(G\to \overline{G}),a,b\rangle = G(a,b)$. In other words, $G = G(a,b)$.
	\begin{prop}
		If $o(b) = rd$, then $\langle \ker(G\to \overline{G}),a \rangle \cap \langle b \rangle = \langle b^d\rangle$. In particular, if $o(b) = d$, then $G(a,b)  = \left(\ker(G\to \overline{G}) \rtimes \langle a \rangle\right) \rtimes \langle b\rangle. $
	\end{prop}
	\begin{proof}
		Let $b^j \in \langle \ker(G\to \overline{G}),a \rangle \cap \langle b \rangle $. It is easy to see that elements of the form $ha^i\in\langle \ker(G\to \overline{G}),a\rangle = \ker(G\to \overline{G}) \rtimes \langle a\rangle$, where $i\neq 0$, have order $p$ so they cannot be in $\langle b\rangle$. Therefore, an element of the intersection $\langle \ker(G\to \overline{G}),a \rangle \cap \langle b \rangle$ must be in $\ker(G\to \overline{G})$. We have $b^j\in \ker(G\to \overline{G})$ if and only if $\beta^j = \overline{1}$, which can only happen when $d\mid j$. Thus, $b^j \in \langle b^d\rangle$. 
		
		If $o(b) = d$, then the second part of the proposition is trivial.
	\end{proof}
	The following result is an immediate consequence of the previous proposition.
	\begin{cor}
		$|G(a,b)|=|\ker(G\to\overline{G})|pd$.
	\end{cor}
	
	\begin{lem}
		An element of $\ker(G\to \overline{G})$ has order dividing $6$.
	\end{lem}
	\begin{proof}
		Let $g\in \ker\left(G\to \overline{G}\right)$. Since $\ker(G\to \overline{G})$ fixes each element of $\mathcal{B}$ setwise, the restriction of $g\in K$ onto a block has order $1,2$, or $3$. Therefore, $o(g) \mid  6$.
	\end{proof}

	We can assume, without loss of generality, that $b$ fixes an element in $B_1 \in \mathcal{B}$ since if the restriction of $b$ onto $B_1$ is a $3$-cycle $\sigma$, then we may find $h\in \ker(G\to \overline{G})$ whose restriction onto $B_1$ is also $\sigma^{-1}$. We then replace $b$ by $\sigma b$.
	
	Hence, either $b$ fixes $B_1$ pointwise or we may assume that its restriction on $B_1$ is the permutation $(x_1 \ y_1)$. As we will see in the next section, it is imperative to know whether $b_{|B_1}$ is trivial or a transposition. 
	
	\begin{rmk}
		If $\ker(G\to \overline{G})$ admits an involution, then we may conjugate this involution so that it does not fix all points in $B_1$. By conjugating with an appropriate element of order $3$ in the minimal normal subgroup $N$, we obtain an involution $\sigma$ whose restriction on $B_1$ is $(x_1 \ y_1)$. As $\sigma \in \ker(G \to \overline{G})$, we may replace $b$ with $ \sigma b$, as this element fixes $B_1$ pointwise.\label{rmk:fixed-points-involution}
	\end{rmk}

	In the following two sections, we will consider two cases, depending on whether	$\ker(G\to \overline{G})$ has an involution or no.

	\section{The kernel $\ker(G\to \overline{G})$ has no involutions}\label{sect:no-involutions}
	Let $G\leq \sym(\Omega)$ be a group satisfying Assumption~\ref{assumption} and $\ker(G\to \overline{G})$ has no involutions. Since $\ker(G\to \overline{G})$ does not have an involution, every element of $\ker(G\to \overline{G})$ has order $3$, i.e., it is an elementary abelian $3$-group. Assume that $\overline{G} = \langle\alpha \rangle \rtimes \langle \beta \rangle \leq \agl{1}{p}$ is non-cyclic and $\beta \alpha \beta^{-1} = \alpha^t$, for some $t\in \mathbb{Z}$ such that $\gcd(t,p)=1$. Let $a$ be a semi-regular element such that $\overline{a} = \alpha$ and $b\in G$ such that $\overline{b} = \beta$ with $o(\beta) = d$ and $o(b) = rd$. Since $\beta \alpha \beta^{-1} = \alpha^t$, there exists $h\in \ker(G\to \overline{G})$ such that $bab^{-1} = ha^t$. Henceforth, let $G(a,b)$ be the group defined in \eqref{second-group}.
	
	Assume without loss of generality that
	\begin{align}
		a = \left(x_1\ x_2\ \ldots\ x_p\right)(y_1\ y_2\ \ldots \ y_p)(z_1\ z_2\ \ldots \ z_p).
	\end{align}
	Consequently, $\alpha =\overline{a} = (B_1 \ B_2\ \ldots \ B_p)$.
	As 
	$\beta\alpha\beta^{-1} = \alpha^t$, we know that $t^d \equiv 1 \pmod p$. Hence, we must have for any $i\in \{2,3,\ldots,p\}$ that
	\begin{align*}
		B_i^\beta = B_{1+(i-1)t}.
	\end{align*} 
	Therefore, the cycle of $\beta$ containing $B_i$ must be of the form
	\begin{align*}
		\left(B_i\ B_{1+(i-1)t}\ B_{1+(i-1)t^2}\ \ldots B_{1+(i-1)t^{d-1}} \right).
	\end{align*}
	
	The structure of the derangement graph of $G(a,b)$ depends on the number of fixed points of the restriction $b_{|B_1}$ of $b$ onto $B_1$.  We distinguish the cases whether $b$ fixes $1$ or $3$ points of $B_1$.
	
	\subsection{Case~1. $b$ fixes $B_1$ pointwise}
	Throughout this subsection, we make the assumption that $b$ fixes $B_1$ pointwise.
	
	\begin{lem}
		For any $u,v\in \mathbb{Z}$ with $d\nmid v$ (or equivalently, $\overline{b^v} \neq \overline{1}$), the element $a^ub^v$ is conjugate to $k_{u,v}b^v$ in $G(a,b)$, for some $k_{u,v}\in \ker(G\to \overline{G})$.\label{lem:conjugates}
	\end{lem}	
	\begin{proof}
		The two elements $\alpha^u\beta ^v$ and $\beta^v$ are conjugate if and only if there exists $g\in G$ such that $\alpha^u\beta^v = \overline{g}\beta^v\overline{g^{-1}}$. We shall prove that in fact we may find $n\in\mathbb{Z}$ such that $g=a^n$. First, note that since $\beta\alpha\beta^{-1} = \alpha^t$ and $d$ is the order of $\beta$, we must have that $d$ is the smallest positive integer such that $t^d-1$ is divisible by $p$, and $p\mid (t^s-1)$ if and only if $d\mid s$. Therefore, we know that $\gcd(t^v-1,p) = 1$. Let $n$ be the unique solution of $(1-t^v)n \equiv u \pmod {p}$. Then, we have
		\begin{align*}
			\alpha^n\beta^v\alpha^{-n} & = \alpha^{n-nt^v} \beta^v = \alpha^{n(1-t^v)}\beta^v = \alpha^u\beta^v.
		\end{align*}
		Consequently, there exists $k_{u,v}\in \ker(G\to \overline{G})$ such that $a^ub^v$ is conjugate to $k_{u,v}b^v$. This completes the proof.
	\end{proof}
	
	\begin{rmk}	
		Assume that $a^ub^v$ leaves $B_i$ invariant, for some $i\in \{1,2,\ldots,p\}$. 
		Since 
		\begin{align*}
			B_i^{a^ub^v} = B_{i+u}^{b^v} = B_{1+(i+u-1)t^v},
		\end{align*}
		the unique element fixed by $a^ub^v$ setwise is $B_i$, where $i$ is the unique solution modulo $p$ to the modular equation $$(t^v-1)i\equiv (1-u)t^v-1 \pmod{p}.$$
	\end{rmk}
	
	Now, we are ready to prove the first main theorem of this section.	
	\begin{thm}
		If $b$ fixes $B_1$ pointwise, then 
		\begin{align*}
			\rho(G(a,b)) = \max\{1,\tfrac{3}{d}\}.
		\end{align*}
		In particular, if $G$ admits an element $b$ of order at least $3$ fixing $B_1$ pointwise and $\overline{b} = \beta$, then $\rho(G) = 1$.\label{thm:special-case}
	\end{thm}
	\begin{proof}
		Let $K = \ker(G\to \overline{G})$ and let $\Gamma_{G(a,b)}$ be the derangement graph of $G(a,b)$. Assume that $|K| = m$. By Assumption~\ref{assumption}, $K$ is derangement-free, so it is a coclique of the derangement graph $\Gamma_{G(a,b)}$. Recall that a right-transversal is a system of distinct representatives of right cosets of $K$. The set $\{ a^ub^v: u\in \{0,1,\ldots,p-1\} \mbox{ and } v\in \{0,1,\ldots,rd-1\} \}$ is not a right-transvesal of $K$ in $G(a,b)$ since $b^d\in \ker(G\to \overline{G})$. It is not hard to see however that $\{ a^ub^v: u\in \{0,1,\ldots,p-1\} \mbox{ and } v\in \{0,1,\ldots,d-1\} \}$ is a right-transversal of $K$ in $G(a,b)$.
		
		\begin{claim}
		{If $v = v^\prime$ and $u\neq u^\prime$, then the subgraph of $\Gamma_{G(a,b)}$ induced by $Ka^ub^v$ and $Ka^{u^\prime}b^{v}$ is a complete bipartite graph $K_{m,m}$.}\label{claim:claim1}
		\end{claim}
		
		\begin{proof}[Proof of Claim~1]
			Let $k,k^\prime \in K.$ Then, $ka^ub^v(k^\prime a^{u^\prime}b^v)^{-1} = ka^{u-u^\prime} \left(k^\prime\right)^{-1} = k^{\prime\prime}a^{u-u^\prime}$, for some $k^{\prime\prime} \in K$. As $u\neq u^\prime$, it follows that $k^{\prime\prime}a^{u-u^\prime}$ is a derangement. Hence, the subgraph of $\Gamma_{G(a,b)}$ induced by $Ka^{u}b^v \cup Ka^{u^\prime}b^v$ is a complete bipartite graph. 
		\end{proof}
		
		\begin{claim}
		 If $v\neq v^\prime$, then the subgraph of $\Gamma_{G(a,b)}$ induced by $Ka^ub^v\cup Ka^{u^\prime}b^{v^\prime}$ is the lexicographic product $X[\overline{K_{\frac{m}{3}}}]$, where $X$ is the complete bipartite graph $K_{3,3}$ with a perfect matching removed (see Figure~\ref{fig1}).\label{claim2}
		\end{claim}
		
		\begin{proof}[Proof of Claim~2]
			Let $B_i \in \mathcal{B}$ be the block fixed by $$\alpha^{u^\prime}\beta^{v^\prime} (\alpha^{u}\beta^{v})^{-1} = \alpha^{u^\prime}\beta^{v^\prime - v}\alpha^{-u} = \alpha^{u^\prime-ut^{v^\prime-v}}\beta^{v^\prime-v}.$$ 
			Let $K_{x_i} = \{ k\in K:\ x_i^k = x_i \}$ and let $c\in K$ such that the restriction of $c$ onto $B_i$ is $c_{|B_i} = (x_i\ y_i\ z_i)$. Clearly, $\langle c\rangle$ is a right-transversal of $K_{x_i}$ in $K$. Therefore, $K=K_{x_i} \cup K_{x_i}{c} \cup K_{x_i}{c^2}$ is a disjoint union.
			
			Now consider two arbitrary elements $kc^\ell a^ub^v\in Ka^ub^v$ and $k^\prime c^{\ell^\prime} a^{u^\prime}b^{v^\prime}\in Ka^{u^\prime}b^{v^\prime}$, for some $\ell,\ell^\prime \in \{0,1,2\}$.
			We claim that if these two elements intersect on $\omega \in \Omega$, then $\omega\in B_i$. Indeed, $nc^\ell a^ub^v$ and $n^\prime c^{\ell^\prime} a^{u^\prime}b^{v^\prime}$ intersect on $\omega\not \in B_i$ if and only if $nc^\ell $ and $n^\prime c^{\ell^\prime} a^{u^\prime}b^{v^\prime}(a^ub^v)^{-1}$ also agree on $\omega$. As $nc^\ell,n^\prime c^{\ell^\prime} \in K$ and $B_i$ is the unique block fixed by $a^{u^\prime}b^{v^\prime}(a^ub^v)^{-1}$ setwise, we conclude that an element on which $kc^\ell $ and $k^\prime c^{\ell^\prime} a^{u^\prime}b^{v^\prime}(a^ub^v)^{-1}$ agree must be in $B_i$.
			
			Now, we note that
			\begin{align*}
				a^{u^\prime}b^{v^\prime}(a^ub^v)^{-1} = z a^{u^\prime -ut^{(v^\prime-v)}} b^{(v^\prime -v)}
			\end{align*}
			for some $z\in \ker(G\to \overline{G})$. By Lemma~\ref{lem:conjugates}, $a^{u^\prime -ut^{(v^\prime-v)}} b^{(v^\prime -v)}$ is conjugate to $kb^{(v^\prime - v)}$, for some $k\in \ker(G\to \overline{G})$. If $k_{|B_1}$ is trivial, then the permutation $\left(a^{u^\prime -ut^{(v^\prime-v)}} b^{(v^\prime -v)}\right)_{|B_i}$ is trivial, whereas $\left(a^{u^\prime -ut^{(v^\prime-v)}} b^{(v^\prime -v)}\right)_{|B_i}$ is a $3$-cycle if $k_{|B_1}$ has order $3$. Therefore, $\left(a^{u^\prime}b^{v^\prime}(a^ub^v)^{-1}\right)_{|B_i}$ is of order $1$ or $3$.
			
			Note that $n,n^\prime \in K_{x_i}$ fix $B_i$ pointwise. Hence, $nc^\ell$ and $n^\prime c^{\ell^\prime} a^{u^\prime}b^{v^\prime}(a^ub^v)^{-1}$ are intersecting if and only if $c^\ell$ and $c^{\ell^\prime} a^{u^\prime}b^{v^\prime}(a^ub^v)^{-1}$ are intersecting on an element of $B_i$. Since $a^{u^\prime}b^{v^\prime}(a^ub^v)^{-1}$ is of order $1$ or $3$, we have that $c^\ell$ and $c^{\ell^\prime} a^{u^\prime}b^{v^\prime}(a^ub^v)^{-1}$ are intersecting on $B_i$ for a unique $s \in \{0,1,2\}$ such that $\ell-\ell^\prime = s$.
			
			Since there are only three choices for $\ell$ and $\ell^\prime$, we conclude overall that the subgraph of $\Gamma_{G(a,b)}$ induced by $Ka^ub^v$ and $Ka^{u^\prime}b^{v^\prime}$ is equal to the lexicographic product  $X[\overline{K_{\frac m3}}]$, where $X$ is the graph in Figure~\ref{fig1}.
			This completes the proof of Claim~\ref{claim2}.\qedhere
		\end{proof}
		
		Hence, the derangement graph of $G(a,b)$ is isomorphic to $\Gamma_{p,d}^{3,\Sigma}(m)$, for some multiset of permutations $\Sigma$ of $\sym(3)$. We conclude that
		\begin{align*}
			\rho(G(a,b)) &= \max\left\{\tfrac{3p|K|}{|G(a,b)|},\tfrac{3pd|K|}{3|G(a,b)|}\right\} = \max\left\{ \tfrac{3}{d}, 1\right\}.\qedhere
		\end{align*}
	\end{proof}

	\subsection{Case~2. $b$ has one fixed point on $B_1$} We will show that whenever $b_{|B_1}$ is a transposition, then the situation is quite different to the previous subsection. We first note that $o(b) = rd$ is even since $b_{|B_1}$ is a transposition. 
	
	{First assume that $d$ is odd. Then, $b^2$ fixes $B_1$ pointwise and $\overline{b^2} = \beta^2$. Hence, $o\left(\beta^2\right)=\frac{d}{\gcd(d,2)} = o(\beta)$. From this, we deduce that $\langle \beta^2\rangle = \langle \beta\rangle$, which in turns implies that $b \in \ker(G\to \overline{G}) \langle b^2 \rangle$. Further, one can easily deduce that $G(a,b) = G(a,b^2)$. As $b^2$ fixes $B_1$ pointwise, we can use Theorem~\ref{thm:special-case} to show that 
	\begin{align*}
		\rho(G(a,b)) = \rho(G(a,b^2)) = \max\{1,\tfrac{3}{d}\}.
	\end{align*}}
	
	We therefore assume that $d$ is even for the remainder of this section. We will first state some results that turn out to be useful to the proof of the main result.
	
	The proof of the following lemma is omitted since it is similar to Lemma~\ref{lem:conjugates}.
	\begin{lem}
		For any $u,v\in \mathbb{Z}$ with $d\nmid v$, the element $a^ub^v$ is conjugate to $k_{u,v}b^v$ in $G(a,b)$, for some $k_{u,v}\in \ker(G\to \overline{G})$.\label{lem:conjugates-invo}
	\end{lem}	
	We derive the following corollary to this lemma.
	\begin{cor}
		Let $u\in \{0,1,\ldots,p-1\}$ and $v\in \{1,2,\ldots,d-1\}$. Assume that $B_i$ is fixed by $\alpha^u\beta^v$. If $v$ is even, then $a^ub^v$ fixes the same number of points as $k_{u,v}$ (see the statement of Lemma~\ref{lem:conjugates-invo}) in $B_i$. If $v$ is odd, then $a^ub^v$ fixes a unique element from $B_i$.
		\label{cor:fixed-points}
	\end{cor}
	
	\begin{thm}
		If $b_{|B_1}$ is a transposition and $d=o(\beta)$ is even, then $\rho(G(a,b)) = \max\{1,\frac{6}{d}\}$. In particular, if $d\leq 6$, then $\rho(G(a,b))\in \{1,\frac{3}{2},3\}$.\label{thm:special-case-2}
	\end{thm}
	\begin{proof}
		Again, let $K = \ker(G\to \overline{G})$ and let $\Gamma_{G(a,b)}$ be the derangement graph of $G(a,b)$. Assume that $|K| = m$. By Assumption~\ref{assumption}, $K$ is derangement-free, so it is a coclique of the derangement graph $\Gamma_{G(a,b)}$. Consider the right-transversal of $K$ in $G(a,b)$ given by $\{ a^ub^v: u\in \{0,1,\ldots,p-1\} \mbox{ and } v\in \{0,1,\ldots,d-1\} \}$.
		
		The proof of the following claim is omitted since it is similar to its analogue in Theorem~\ref{thm:special-case}.
		
		\begin{claim}
		If $v=v^\prime$ and $u\neq u^\prime$, then the subgraph of $\Gamma_{G(a,b)}$ induced by $Ka^ub^v\cup Ka^{u^\prime}b^{v}$ is a complete bipartite graph $K_{m,m}$.
		\end{claim}
		
		\begin{claim} 
		 If $v\neq v^\prime$, then the subgraph of $\Gamma_{G(a,b)}$ induced by $Ka^ub^v$ and $Ka^{u^\prime}b^{v^\prime}$ is empty if $v^\prime-v$ is odd, and equal to $X[\overline{K_{\frac{m}{3}}}]$ where $X$ is the graph in Figure~\ref{fig1} if $v^\prime-v$ is even.
		\end{claim}
		\begin{proof}[Proof of Claim~4]
			Let $B_i$ be the unique block fixed by $\alpha^{u^\prime}\beta^{v^\prime}(\alpha^u\beta^v)^{-1}$. Let $c \in K$ such that $\langle c\rangle$ is a right-transversal of $K_{x_i}$ in $K$. Similar to the proof of Claim~\ref{claim2} in Theorem~\ref{thm:special-case}, we examine the edges between $K_{x_i}c^\ell a^{u}b^{v}$ and $K_{x_i}c^{\ell^\prime} a^{u^\prime}b^{v^\prime}$, for all $\ell,\ell^\prime \in \{0,1,2\}$. For any $k,k^\prime \in K_{x_i}$, we have that $kc^\ell a^{u}b^{v}$ and $k^\prime c^{\ell^\prime} a^{u^\prime}b^{v^\prime}$ are intersecting if and only if  $kc^\ell $ and $k^\prime c^{\ell^\prime} a^{u^\prime}b^{v^\prime}\left(a^{u}b^{v}\right)^{-1}$ are intersecting on $B_i$. Here, it is worthwhile to note that
			\begin{align*}
				a^{u^\prime}b^{v^\prime}\left(a^{u}b^{v}\right)^{-1} = za^{u^\prime -ut^{(v^\prime -v)}} b^{(v^\prime  -v)},
			\end{align*}
			for some $z\in K$.
			\begin{itemize}
			\item If $v-v^\prime$ is even, then $b^{v^\prime-v}$ fixes $B_1$ pointwise, and by Corollary~\ref{cor:fixed-points} we know that the permutation $\left(za^{u^\prime -ut^{(v^\prime -v)}} b^{(v^\prime  -v)}\right)_{|B_i}$ has the same number of fixed points on $B_i$ as 
			\begin{align*}
				\left(zk_{u^\prime -ut^{(v^\prime -v)},v^\prime -v}\right)_{|B_1}.
			\end{align*}
			The latter has $1$ or $3$ fixed points.
			Using the same argument as Claim~\ref{claim2} in the proof of Theorem~\ref{thm:special-case}, we conclude that the subgraph induced by $Ka^ub^v$ and $Ka^{u^\prime}b^{v^\prime}$ is isomorphic to $X[\overline{K_{\frac{m}{3}}}]$, where $X$ is the graph given in Figure~\ref{fig1}.
			
			\item If $v^\prime  - v$ is odd, then by Corollary~\ref{cor:fixed-points},  $a^{u^\prime -ut^{(v^\prime -v)}} b^{(v^\prime  -v)}$ fixes a unique point of $B_i$, and so does $za^{u^\prime -ut^{(v^\prime -v)}} b^{(v^\prime  -v)}$ since $z\in K$ has order $1$ or $3$. Let $c \in K$ be such that $\langle c\rangle$ is a right-transversal of $K_{x_i}$ in $K$. Let $\ell,\ell^\prime  \in \{0,1,2\}$. For any $n,n^\prime \in K_{x_i}$, the elements $nc^{\ell}a^ub^v$ and $n^\prime c^{\ell^\prime} a^{u^\prime}b^{v^\prime}$ always agree on an element of $B_i$ since $n,n^\prime \in K_{x_i}$ fix $B_i$ pointwise, and the permutations $c^{\ell}_{|B_i}$ and $\left( c^{\ell^\prime} a^{u^\prime}b^{v^\prime}\left(a^ub^v\right)^{-1}\right)_{|B_i} = \left(c^{\ell^\prime}za^{u^\prime-ut^{(v^\prime -v)}} b^{v^\prime -v}\right)_{|B_i}$ are respectively elements of order belonging to $\{1,3\}$ and an involution. Consequently, there is no edge between $Ka^ub^v$ and $Ka^{u^\prime}b^{v^\prime}$ in $\Gamma_{G(a,b)}$.\qedhere
			\end{itemize}
		\end{proof}
		By combining Claim~3 and Claim~4, it is not hard to show that $\Gamma_{G(a,b)}$ is the union of two graphs isomorphic to $\Gamma_{p,t_1}^{3,\Sigma}(|K|)$ and $\Gamma_{p,t_2}^{3,\Sigma}(|K|)$, where 
		\begin{align*}
			t_1 = \left|\left\{ 0\leq i\leq d-1: i\mbox{ is odd } \right\}\right| =  \tfrac{d}{2}  \mbox{ and }t_2 = \left|\left\{ 0\leq i\leq d-1: i\mbox{ is even } \right\}\right| =  \tfrac{d}{2} . 
		\end{align*}
		The independence number of the graph $\Gamma_{p,t_1}^{3,\Sigma}(|K|)$ and $\Gamma_{p,t_2}^{3,\Sigma}(|K|)$ are both equal to $\max\{ t_1|K_{x_1}|,|K| \} = \max\{ t_2|K_{x_1}|,|K|\}$. The independence number of the union of the two graphs is $$\max\{ |K_{x_1}|d,2\max\{\tfrac{d}{2}|K_{x_1}|,|K|\}\}.$$
		Therefore, $$\rho(G(a,b)) =  \max\left\{1, \max\left\{\tfrac{6}{d},1\right\}\right\} = \max\left\{1,\tfrac{6}{d}\right\}.$$
		
		If $ \tfrac{d}{2} \geq 3$, then $\rho(G(a,b)) = 1$. If $\tfrac{d}{2} < 3$, then using the fact that $d$ is even we have $d\in \{2,4\}$  and $\rho(G(a,b)) \in \{\frac{3}{2},3\}$.
		This completes the proof.
	\end{proof}
	\section{$\ker(G\to \overline{G})$ has an involution}\label{sect:involutions}
{	Throughout this section, we let $G\leq \sym(\Omega)$ be a transitive group satisfying Assumption~\ref{assumption} and we assume that $\ker(G\to \overline{G})$ admits an involution. Assume that $\overline{G} = \langle \alpha\rangle\rtimes \langle \beta\rangle$, 
	and recall that $a$ and $b$ are two elements of $G$ such that $\overline{a} = \alpha$ and $\overline{b} = \beta$. Recall that $o(\alpha) = o(a) = p$, $o(\beta) = d$, and $o(b) = rd$, for some positive integer $r$. Since $\ker(G\to\overline{G})$ has an involution, we may assume that $b$ fixes $B_1$ pointwise (See Remark~\ref{rmk:fixed-points-involution}). As we have seen previously, we have $G = G(a,b)$.
	
	For any distinct $ g,g^\prime\in \ker(G\to \overline{G})$ of order $3$, we have $gg^\prime g^{-1} = g^\prime$, and thus $\langle g,g^\prime \rangle = C_3\times C_3$ unless $g^\prime = g^{-1}$. Let $E$ be the subgroup generated by all elements of order $3$ in $\ker(G\to \overline{G})$. By commutativity of the elements of order $3$ in $\ker(G\to \overline{G})$, we know that $E$ is an elementary abelian $3$-group.
	
	Again, we let $K = \ker(G\to \overline{G})$ and $|K| = m$. The set $\{ Ka^ub^v :\ u\in \{0,1,\ldots,p-1\}, v \in \{0,1,\ldots,d-1\} \}$ is again a right-transversal of $K$ in $G(a,b)$. For any $u,u^\prime \in \{0,1,\ldots,d-1\}$ and $v,v^\prime \in \{0,1,\ldots,d-1\}$, we have 
	\begin{align*}
		\left(Ka^ub^v\right) \left(Ka^{u^\prime} b^{v^\prime}\right) = K a^{u^\prime -ut^{v^\prime -v}} b^{v^\prime -v}.
	\end{align*}
	\begin{rmk} 
		\hfil
		\begin{enumerate}[(a)]
			\item For any $u\in \{0,1,\ldots,p-1\}$ and $v\in \{1,\ldots,d-1\}$, the element $a^ub^v$ fixes a certain block $B_i$ setwise. Since $K$ admits an involution, we may assume that $a^ub^v$ fixes this block pointwise, otherwise we can replace the representative with something fixing the block pointwise (by multiplying with an element of $K$).\label{rmk:a}
			\item In contrast to \eqref{rmk:a}, the case where $\ker(G\to \overline{G})$ has no involution is quite different. Indeed, if $\left(a^ub^v\right)_{|B_i}$ is an involution, then one cannot multiply it with an element of $\ker(G\to \overline{G})$ to make the resulting permutation fix $B_i$ pointwise.
	\end{enumerate}\label{rmk:important} 
	\end{rmk} 
	Now, we proceed with the proof. Let $\{k_1=1,k_2,\ldots,k_{s}\}$ be a right-transversal of $E$ in $K$, where $s$ is the index of $E$ in $K$. Then,
	\begin{align*}
		K = \bigcup_{j = 1}^s Ek_j
	\end{align*}
	Also, let $c\in E$ such that $c_{|B_i}\neq 1$. Then, 
	\begin{align*}
		E = \bigcup_{\ell = 0}^2E_{x_i}c^\ell.
	\end{align*}
	Hence,
	\begin{align*}
		K = \bigcup_{j=1}^s \bigcup_{\ell =0}^{2}  E_{x_i}c^\ell k_j.
	\end{align*}
	
	In the next lemma, we show that  every non-trivial element of the right-transversal $\{k_1=1,k_2,\ldots,k_{s}\}$ can be assumed to have order $2$.
	\begin{lem}
		The exists a right-transversal of the subgroup $E$ of $K$ consisting of the identity and involutions.
	\end{lem}
	\begin{proof}
		Recall that $\{k_1=1,k_2,\ldots,k_{s}\}$ is a right-transversal. Since $E$ contains all elements of order $3$ of $K$, $o(k_i)\in \{2,6\}$ for $i\neq 1$. If $o(k_i) = 2$, then we are done. If $o(k_i) = 6$, then we know that $k_i = k_i^{-2}k_i^3$, and that $o(k_i^{-2}) = 3$ and $o(k_i^{3}) = 2$. Consequently, we conclude that $ k_i\in Ek_i^3$, so $Ek_i = Ek_i^3$. In other words, there exists a right-transversal consisting of the identity and involutions.
	\end{proof}
	
	From the above lemma, we assume henceforth $\{1,k_2,\ldots,k_s\}$ is such that $o(k_2) =\ldots =  o(k_s) = 2$.
	\begin{thm}
		$\rho(G(a,b)) = \max \left\{ 1,\frac{3}{d} \right\}$.\label{thm:special-case-3}
	\end{thm}
	\begin{proof}
		 
		  Let $u,u^\prime\in \{0,1,2,\ldots,p-1\}$ and $v,v^\prime \in \{0,1,\ldots,d-1\}$, and let us determine the edges induced by the vertices in $Ka^ub^v \cup Ka^{u^\prime}b^{v^\prime}$ in $\Gamma_{G(a,b)}$. 
		
		The proof of the following claim is omitted since it is similar to the proof of Claim~\ref{claim:claim1}.
		
		\begin{claim}
		If $v = v^\prime$ and $u\neq u^\prime$, then the subgraph of $\Gamma_{G(a,b)}$ induced by $Ka^ub^v \cup Ka^{u^\prime} b^{v^\prime}$ is a complete bipartite $K_{m,m}$.
		\end{claim} 
		
		Suppose that $B_i$ is the unique block fixed setwise by $a^{u^\prime}b^{v^\prime}\left(a^ub^v\right)^{-1}$.
		\begin{claim} 
		 If $v \neq v^\prime$, then the subgraph of $\Gamma_{G(a,b)}$ induced by $Ka^ub^v \cup Ka^{u^\prime} b^{v^\prime}$ contains a subgraph isomorphic to the disjoint union of $s = [K:E]$ copies of $X[\overline{K_{\frac{|E|}{3}}}]$, where $X$ is the graph in Figure~\ref{fig1}.
		\end{claim}
		\begin{proof}[Proof of Claim~6]
			Recall that $B_i$ is the unique element of $\mathcal{B}$ fixed by $(a^ub^v)(a^{u^\prime}b^{v^\prime})^{-1}$ setwise. By Remark~\ref{rmk:important}, we may assume that $(a^ub^v)(a^{u^\prime}b^{v^\prime})^{-1}$ fixes $B_i$ pointwise. 
			 We note that a vertex in $Ka^{u}b^{v}$ intersects a vertex in $Ka^{u^\prime}b^{v^\prime}$ on $\omega\in \Omega$ only if $\omega \in B_i$.

			Fix $j\in \{1,2,\ldots,s\}$. We will show now that for any $\ell\in \{0,1,2\}$, there exists a unique $\ell^\prime \in \{0,1,2\}$ such that the subgraph of $\Gamma_{G(a,b)}$ induced by $Ec^\ell k_j \cup Ec^{\ell^\prime}k_j$ is isomorphic to $X[\overline{K_{\frac{|E|}{3}}}]$. Since the elements of $E_{x_i}$ fix $B_i$ pointwise, we first note that one only needs to determine whether $c^\ell k_j$ and $c^{\ell^\prime}k_{j}(a^ub^v)(a^{u^\prime}b^{v^\prime})^{-1}$ are adjacent, or equivalently do not fix a point in $B_i$. As $(a^ub^v)(a^{u^\prime}b^{v^\prime})^{-1}$ also fixes $B_i$ pointwise, we only need to check the adjacency between $c^\ell k_j$ and $c^{\ell^\prime}k_{j}$. Hence, for any $\ell \in \{0,1,2\}$ there exists a unique $s\in \{0,1,2\}$ such that $\ell -\ell^\prime =s$ and for which vertices in $E_{x_i}c^\ell k_j a^ub^v \cup E_{x_i}c^{\ell^\prime} k_j a^{u^\prime}b^{v^\prime}$ form a coclique. For any other values of $\ell-\ell^\prime$, the vertices in $E_{x_i}c^\ell k_j a^ub^v \cup E_{x_i}c^{\ell^\prime} k_j a^{u^\prime}b^{v^\prime}$  induces a complete bipartite graph.
			
			Consequently, the subgraph of $\Gamma_{G(a,b)}$ induced by $ Ek_ja^ub^v \cup Ek_ja^{u^\prime}b^{v^\prime}$ is isomorphic to $X[\overline{K_{\frac{|E|}{3}}}]$. The $s$ copies are obtained by varying $j\in \{1,2,\ldots,s\}$. This completes the proof of Claim~6.
		\end{proof} 
		\begin{claim}
		There is a homomorphism from $\Gamma_{p,d}^{\Sigma}(|E|)$ to $\Gamma_{G(a,b)}$, for some multiset $\Sigma$ of permutations of $\sym(3)$.\label{claim:claim7}
		\end{claim} 
		\begin{proof}[Proof of Claim~7]
			Using Claim~6, it is easy to see that there is in fact a multiset of permutation $\Sigma$ of $\sym(3)$ for which the graph $\Gamma_{p,d}^\Sigma(|E|)$ can be embedded into $\Gamma_{G(a,b)}$.
		\end{proof}
		
		Using Claim~\ref{claim:claim7} and the No-Homomorphism Lemma, we conclude that 
		\begin{align*}
			1\leq \rho(G(a,b)) = \frac{\alpha(\Gamma_{G(a,b)})}{\frac{|G(a,b)|}{3p}}\leq \frac{\alpha(\Gamma_{p,d}^\Sigma(|E|))}{\frac{|E|pd}{3p}} = \frac{\max\{|E_{x_1}|d,|E|\}}{|E_{x_1}|d} = \max\{1,\tfrac{3}{d}\}.
		\end{align*}
		If $d\geq 3$, then clearly $\rho(G(a,b)) = 1$. If $1\leq d\leq 2$, then $\frac{3}{d} = \frac{|K|}{|K_{x_1}|d}$ is attained through the intersecting set $K$. Thus, if $1\leq d\leq 2$, then $\rho(G(a,b)) = \frac{3}{d}$. This completes the proof of the theorem.
	\end{proof}}
	
	\section{The Doubly transitive case}\label{sect:2-transitive}
	Throughout this section, we assume the following.
	\begin{ass}
		Let $G\leq \sym(\Omega)$ is a transitive group satisfying Assumption~\ref{assumption} such that $\overline{G}$ is $2$-transitive. \label{assumption2}
	\end{ass}
	  As $\overline{G}$ has degree $p$, the theory of transitive groups of prime degrees plays an important role in what follows. The next lemma is crucial to the proof of the main result of this section.
	\begin{lem}
		If $H\leq\sym(p)$ is transitive and $P$ is a Sylow $p$-subgroup of $H$, then $P$ is cyclic and $\operatorname{N}_{H}(P)=P$ if and only if $H= P$.\label{lem:cyclic}
	\end{lem}
	\begin{proof}
		A Sylow $p$-subgroup of $H$ has order $p^k$, for some $k\geq 1$. Since $p^2\nmid p!$, clearly a Sylow $p$-subgroup of any transitive subgroup of $\sym(p)$ must be cyclic and thus of order $p$. 
		
		Let $P$ be a Sylow $p$-subgroup of $H$. If $H = P$, then $\operatorname{N}_H(P) = H$. Conversely, if $\operatorname{N}_H(P) = P$, then $P$ is the unique Sylow $p$-subgroup of $H$ and $P\trianglelefteq H$. Hence, $P\leq H\leq \operatorname{N}_H( P)= P$, which completes the proof.
	\end{proof}
	
	The main result of this section is the following.
	\begin{thm}
		If $G\leq \sym(\Omega)$ satisfies Assumption~\ref{assumption2}, then $\rho(G) = 1$ unless $p = q+1$ is a Fermat prime and $\psl{2}{q}\leq \overline{G} \leq \pgammal{2}{q}$.\label{thm:special-case-4}
	\end{thm}
	\begin{proof}
	As $\overline{G}$ is $2$-transitive, it is one of the groups described in \eqref{trans1}-\eqref{trans5}. Let $P\leq \overline{G}$ be a Sylow $p$-subgroup. By Lemma~\ref{lem:cyclic}, $P$ is cyclic. Since $\overline{G}$ is $2$-transitive, the subgroup $P$ cannot be self-normalizing. Therefore, $\operatorname{N}_{\overline{G}}\left(P\right) = P\rtimes Q$, for some cyclic subgroup $Q\leq \overline{G}$ such that $|Q|= d\neq 1$ and $d \mid (p-1)$. 
	
	\begin{claim}
		$Q$ has an element of order at least $3$, unless $\psl{2}{q}\leq \overline{G} \leq\pgammal{2}{q} $, where $p=q+1$ is a Fermat prime.\label{claim:last}
	\end{claim}
	\begin{proof}[Proof of Claim~8]
	If $\overline{G} = \agl{1}{p}$, then $\operatorname{N}_{\overline{G}}(P) = P\rtimes Q \cong \agl{1}{p}$, so $Q$ admits an element of odd order if $p-1$ is not a power of $2$. Assume now that $p-1= 2^k$, for some $k\geq 1$. If $k\geq 2$, then it admits an element of order larger than $3$. If $k = 1$, then $p = 3$ contradicts the fact that $p\geq 5 $. This settles \eqref{trans1}.
	
	If $\overline{G} = \alt(p)$, then $\operatorname{N}_{\overline{G}}(P) \cong C_{p} \rtimes C_{\frac{p-1}{2}}$. Similar to the previous paragraph, if $\frac{p-1}{2}$ is not a power of $2$, then $Q\leq \operatorname{N}_{\overline{G}}(P)$ admits the desired element. If $p-1 = 2^k$, for some $k\geq 3$, then a similar result holds. If $k = 2$, then $Q$ does not have an element of order larger than $2$. However, this case can be omitted from the analysis since no transitive group of degree $15$ satisfy Assumption~\ref{assumption}, and we already know that $\mathcal{I}_{15} = \{1\}$. This settles \eqref{trans2}.
	
	For \eqref{trans3}, the normalizers of a Sylow $11$-subgroup of $\psl{2}{11}$ and $\mathieu{11}$ are both isomorphic to $C_{11}\rtimes C_5$.  Similarly, the normalizer of the groups in \eqref{trans4} are both isomorphic to $C_{23}\rtimes C_{11}$.
	
	Finally, we consider the socle in \eqref{trans5}, that is, $\psl{n}{q}$ for some prime number $n$ and a prime power $q$, such that $p=\frac{q^n-1}{q-1}$. It is well known that $\psl{n}{q}$ admits a Singer cycle $A$ of order $p$. Since a Singer subgroup (i.e., a subgroup generated by a Singer cycle) is isomorphic to a Sylow $p$-subgroup in this case, the normalizer of the Singer cycle $\langle A\rangle$ is a Frobenius group. Then by \cite{hestenes1970singer}, we must have that $n$ is an odd prime, or $n =2$ and $4\nmid (q+1)$. Assume that $n$ is an odd prime. If $\Phi \in \Aut(\mathbb{F}_{q^n}/\mathbb{F}_q)$ is the Frobenius automorphism, then $\Phi$ induces a collineation $B_\Phi$ of $\pg{n-1}{q}$. Then, $o(B_\Phi) = n$, and $\operatorname{N}_{\psl{n}{q}}(\langle A \rangle) = \langle A\rangle \rtimes \langle B_\Phi \rangle$. Thus, $Q$ admits an element of order $n\geq 3$.
	
	 If $n = 2$, then it is well known that the normalizer of a cyclic group of order $p = q+1$ is isomorphic to $\dih{p}$, so $|Q|=2$.
	\end{proof}
	 Now, let $M\leq G$ be such that $\overline{M} = P\rtimes Q = \langle\alpha\rangle \rtimes \langle \beta \rangle$. Clearly, $M$ is transitive since $\ker(G\to\overline{G})\neq 1$, so $\rho(G)\leq \rho(M)$. If $a,b\in G$ such that $\overline{a} = 
	 \alpha\in P$ and $\overline{b} = \beta\in Q$, then 
	\begin{align*}
		\rho(G)\leq \rho(M) = \rho(M(a,b)) = \max \{ 1,\tfrac{3}{|Q|}\}.
	\end{align*}
	
	Therefore, we only need to show that $Q\leq \operatorname{N}_{\overline{G}}(P)$ contains an element of order at least $3$ normalizing $P$ to show that $\rho(G) = 1$. By Claim~\ref{claim:last}, we conclude that $\rho(G) = 1$ unless $\psl{2}{q} \leq \overline{G} \leq \pgammal{2}{q}$, where $p=q+1$ is a Fermat prime, in which case $1\leq \rho(G)\leq \frac{3}{2}$.
	\end{proof}

	\section{Concluding remarks}\label{sect:conclusion}
	In this paper, we showed that if $p$ is an odd prime, then for any imprimitive group $G$ of degree $3p$ which is not quasiprimitive (i.e.,  admitting a non-trivial and intransitive normal subgroup), $\rho(G) \in \{1,\frac{3}{2},3\}$, unless possibly when $p = 2^{2^k}+1$ is a Fermat prime and the induced action of $G$ on the unique $G$-invariant partition of $\Omega$ is an almost simple group containing a subgroup isomorphic to $\psl{2}{2^{2^k}}$. For the aforementioned case, we can only give an upper bound of $\frac{3}{2}$ on the intersection density of $G$. We are inclined to believe that the intersection density of these groups arising from Fermat primes and $\psl{n}{q}$ are equal to $1$. Thus, we pose the following question.
	\begin{qst}
		Let $p = q+1$ be a Fermat prime. Let $G\leq \sym(\Omega)$ be a transitive group of degree $3p$ satisfying Assumption~\ref{assumption} such that $\psl{2}{q} \leq \overline{G} \leq \pgammal{2}{q}$. Is it true that $\rho(G)=1$?\label{prob}
	\end{qst}

	The results in this paper are further evidence towards the veracity of Meagher's question in Question~\ref{qst:karen}. Provided that Question~\ref{prob} is affirmative, the only cases left to check are the quasiprimitive cases. It was proved in \cite{behajaina2022intersection} that the only quasiprimitive groups of degree $3p$ whose intersection densities are possibly larger than $1$ are almost simple groups with socle equal to $\psl{n}{q}$, where $n$ is a prime, $q$ is a prime power, and $p = \frac{q^n-1}{q-1}$ is an odd prime. In this case, $\ker(G\to \overline{G}) = 1$, so we cannot apply the arguments used in this paper anymore. 
	
	\vspace*{0.5cm}
	\noindent
	{\sc Acknowledgement:} The authors are supported in part by the Ministry of Education, Science and Sport of Republic of Slovenia (University of Primorska Developmental
	funding pillar).
	
	
\end{document}